\documentclass[11pt]{article}
\usepackage{amssymb,amsfonts, amsmath,amsthm}
\usepackage[mathscr]{euscript}
\usepackage[utf8]{inputenc}
\usepackage{caption}
\usepackage{subcaption}
\usepackage{authblk}
\usepackage{color}
\usepackage{esvect}
\usepackage{empheq} 
\usepackage{float}
\usepackage{graphicx}
\usepackage{afterpage}
\usepackage{mwe}
\usepackage{comment}
\usepackage{mathtools}
\usepackage{xcolor}
\usepackage{epsfig,color,graphicx,graphics}
\usepackage[margin=1in]{geometry}
\usepackage{graphicx}

\numberwithin{equation}{section}
\allowdisplaybreaks
\newtheorem{thm}{Theorem}[section]
\newtheorem{defn}[thm]{Definition}
\newtheorem{cor}[thm]{Corollary}
\newtheorem{lemma}[thm]{Lemma}
\newtheorem{rmrk}[thm]{Remark}
\newtheorem{crl}[thm]{Corollary}

\newtheorem{exmp}[thm]{Example}

\newcommand{\e}{\varepsilon}
\newcommand{\R}{\mathbb{R}}
\newcommand{\N}{\mathbb{N}}

\newtheorem{prop}[thm]{Proposition}

\newcommand{\abs}[1]{\left\vert{#1}\right\vert}

\newcommand{\ba}{\begin{array}}
\newcommand{\ea}{\end{array}}

\newcommand{\bthm}{\begin{thm}}
\newcommand{\ethm}{\end{thm}}
\newcommand{\bstp}{\begin{stp}}
\newcommand{\estp}{\end{stp}}
\newcommand{\blemma}{\begin{lemma}}
\newcommand{\elemma}{\end{lemma}}
\newcommand{\bprop}{\begin{prop}}
\newcommand{\eprop}{\end{prop}}
\newcommand{\bpf}{\begin{pf}}
\newcommand{\epf}{\end{pf}}
\newcommand{\bdefn}{\begin{defn}}

\newcommand{\edefn}{\end{defn}}
\newcommand{\brk}{\begin{rmrk}}
\newcommand{\erk}{\end{rmrk}}
\newcommand{\bcrl}{\begin{crl}}
\newcommand{\ecrl}{\end{crl}}
\newcommand{\beg}{\begin{exmp}}
\newcommand{\eeg}{\end{exmp}}
\newcommand{\norm}[1]{\left\|#1\right\|}

\newcommand{\beqn}{\begin{equation}}
\newcommand{\eeqn}{\end{equation}}

\renewcommand{\leq}{\leqslant}
\renewcommand{\geq}{\geqslant}

\newcommand{\mS}{\mathcal{S}}

\newcommand{\beq}{\begin{equation}}
\newcommand{\eeq}{\end{equation}}
\newcommand{\bea}{\begin{eqnarray}}

\newcommand{\eea}{\end{eqnarray}}

\newcommand{\dive}{\mathrm{div}\,}

\def\({\left(}
\def\){\right)}

\def\ba{{\bar\alpha}}

\title{}
\author

\begin{document}
\Large
 \begin{center}
\Large{Asymptotic analysis of  the energy for a ferroelectric nematic.}\\ 

\hspace{10pt}

\large
Dmitry Golovaty$^1$ and Peter Sternberg$^2$ \\

\hspace{10pt}

\small  
$^1$) Department of Mathematics, The University of Akron, Akron, OH \\
dmitry@uakron.edu
\\

$^2$) Department of Mathematics, Indiana University, Bloomington, IN \\
sternber@iu.edu

\end{center}

\hspace{10pt}

\normalsize

\begin{abstract}
The variational model for a ferroelectric nematic bears close resemblance to the well-known energy model for micromagnetics. Despite this similarity, the two models operate in fundamentally distinct parameter regimes describing different physics. In this paper we establish that the ferroelectric nematic energy functional $\Gamma$-converges to the energy of a nematic with high elastic anisotropy.
\end{abstract}

\section{Introduction}

Ferroelectric nematic liquid crystals possess a spontaneous macroscopic
polarization ${\bf{P}}$, coupled to the underlying orientational order of the
nematic. Unlike ordinary nematics, distortions of this polar field generate
bound charge: in the bulk the charge density is proportional to
$-\operatorname{div}{\bf{P}}$, while at the boundary it is proportional to
${\bf{P}}\cdot{\boldsymbol{\nu}}$, where ${\boldsymbol{\nu}}$ denotes the outer unit normal to the sample.
Consequently, splay of ${\bf{P}}$ and a nonzero normal component of ${\bf{P}}$ at the
boundary carry an electrostatic cost. In the presence of mobile ions, this electrostatic interaction is mediated by a screened
potential satisfying a modified Helmholtz equation.

The goal of this paper is to identify the variational limit of this screened
electrostatic energy as the screening length tends to zero. To this end,
let $\Omega$ be a smooth bounded domain in $\R^3$ and, unless specified otherwise, assume that all quantities are nondimensional. Then for a vector polarization ${\bf{P}}\in H^1(\Omega;\mS^2)$ we consider the screened electric potential given by the solution $u_\e:\R^3\to\R$ to the problem
\begin{equation}
\e^2\Delta u-\alpha^2u=\left\{\begin{matrix} \dive {\bf{P}}&\text{in}\;\Omega\\ 0&\text{in}\;\R^3\setminus\overline{\Omega}\end{matrix}\right.\label{Helmholtz}
\end{equation}
subject to the transmission boundary conditions
\begin{equation}
[u]=0\quad\text{and}\quad\left[\frac{\partial u}{\partial {\boldsymbol{\nu}}}\right]=\,-\frac{1}{\e^2}{\bf{P}}\cdot{\boldsymbol{\nu}}\quad\text{along}\;\partial\Omega.\label{bc}
\end{equation}
Here $\e>0$ and $\alpha>0$ are given constants, ${\boldsymbol{\nu}}$ denotes the outer unit normal vector to $\partial\Omega$ and for any quantity $f$, the notation $[\,f\,]$ denotes the difference $f_+-f_-$ between the limiting value of $f$ from outside $\Omega$ and the limiting value from inside $\Omega$. Note that one can think of the parameter $\e/\alpha$ as a nondimensional screening length so that when $\e\to0$ the ferroelectric nematic is subject to strong screening.

For the energy associated with a ferroelectric nematic having polarization ${\bf{P}}\in H^1(\Omega;\mS^2)$ in the presence of a screened electric field $\nabla u_\e$, we take a sum of a nondimensionalized Frank energy and a nondimensionalized electrostatic energy, that is,
\begin{equation}
  E_\e({\bf{P}}):=E_F({\bf{P}})+E_{e\ell}(\phi)\int_{\Omega}\frac{1}{2}\abs{\nabla {\bf{P}}}^2\,dx+\frac{1}{2}\int_{\R^3}\e^2\abs{\nabla u_\e}^2+\alpha^2 u_\e^2\,dx,\label{noGL1}
\end{equation}
where $u_\e$ is required to satisfy \eqref{Helmholtz}-\eqref{bc}, and we have taken the `equal constants' version of the Frank energy for simplicity. 

After an integration by parts, one then finds that $E_\e$ can be rewritten as
\begin{equation}
  E_\e({\bf{P}})=
  \int_{\Omega}\frac{1}{2}\abs{\nabla {\bf{P}}}^2+\frac{1}{2}{\bf{P}}\cdot\nabla u_\e\,dx.\label{noGL}   
\end{equation}

We will also consider an energy allowing for a relaxation of the constraint $\abs{{\bf{P}}}=1$ via a Ginzburg-Landau penalty. For ${\bf{P}}\in H^1(\Omega;\R^3)$ it is given by
\begin{equation}
\tilde{E}_\e({\bf{P}}):=\int_{\Omega}\left\{\frac{1}{2}\abs{\nabla {\bf{P}}}^2+\frac{1}{4\eta^2}(\abs{{\bf{P}}}^2-1)^2 +\frac{1}{2}{\bf{P}}\cdot\nabla u_\e\right\}\,dx,\label{withGL}
\end{equation}
where $\eta\ll 1$ is a relaxation parameter which we will assume to be fixed in the remainder of this paper.

Our main result shows that, in the strongly screened limit
$\varepsilon\to0$, the nonlocal electrostatic interactions entering into $E_\e$ and $\tilde{E}_\e$ become local: adding a splay penalty and imposing a hard tangential boundary condition. We approach this asymptotic limit from the perspective of $\Gamma$-convergence, a variational convergence designed to ensure that limit of minimizers to a sequence of energies converge to a minimizer of the limiting energy, the $\Gamma$-limit. We prove the following:
\begin{thm}\label{GC}
The sequence of energies $\{E_\e\}$ defined for ${\bf{P}}\in H^1(\Omega;\mathbb{S}^2)$ $\Gamma$-converges in the weak topology on $H^1(\Omega;\mathbb{S}^2)$ to the limiting energy
\begin{equation}
E_0({\bf{P}}):=\left\{\begin{matrix} \int_{\Omega}\left\{\frac{1}{2}\abs{\nabla {\bf{P}}}^2+\frac{1}{2\alpha^2}(\dive {\bf{P}})^2\right\}\,dx&\text{if}\;{\bf{P}}\cdot{\boldsymbol{\nu}}=0\;\text{on}\;\partial\Omega\\
+\infty&\;\text{otherwise.}\end{matrix}\right.
\end{equation}
Similarly, the $\Gamma$-limit of the sequence $\{\tilde{E}_\e\}$ defined on $H^1(\Omega;\R^3)$ is
\begin{equation}
\tilde{E}_0({\bf{P}}):=\left\{\begin{matrix} \int_{\Omega}\left\{\frac{1}{2}\abs{\nabla {\bf{P}}}^2+\frac{1}{4\eta^2}(\abs{{\bf{P}}}^2-1)^2+\frac{1}{2\alpha^2}(\dive {\bf{P}})^2\right\}\,dx&\text{if}\;{\bf{P}}\cdot{\boldsymbol{\nu}}=0\;\text{on}\;\partial\Omega\\
+\infty&\;\text{otherwise.}\end{matrix}\right.
\end{equation}
\end{thm}

\begin{rmrk}
 We note that if one begins with the more general form \eqref{genOF} for $E_F$ discussed below, then the limiting problem arising in the $\Gamma$-limit takes the form of augmenting the splay constant $K_1$ by $\frac{1}{2\alpha^2}$.   
\end{rmrk}

The utility of $\Gamma$-convergence is always reliant on an accompanying result on compactness of energy-bounded sequences in the topology of the $\Gamma$-convergence. In our setting, a uniform energy bound on a sequence, 
say $\{\mathcal{{\bf{P}}}_\e\}$, of the form
\[
E_\e(\mathcal{{\bf{P}}}_\e)<C_0\quad\text{or}\quad\tilde{E}_\e(\mathcal{{\bf{P}}}_\e)<C_0
\]
immediately implies a uniform bound on $\norm{\mathcal{{\bf{P}}}_\e}_{H^1(\Omega;\R^3)}$. Hence, such a sequence will always possess a subsequence weakly converging in $H^1(\R^3)$. A standard consequence is the following:
\begin{cor}
    Let $\{{\bf{P}}_\e\}$ denote a sequence of minimizers of $E_\e$ (resp. $\tilde{E}_\e$). Then a subsequence converges to a minimizer of $E_0$ (resp. $\tilde{E}_0$).
\end{cor}

The result provides a variational justification for phenomenological models
of ferroelectric nematics in which splay is assigned a large elastic cost and
the polarization is constrained, either softly or strongly, to be tangential
at the boundary. Such high-splay-penalty models arise naturally in the
description of domain structures, wall formation, and singular patterns in
ferroelectric nematic films. The present theorem shows that this local
description can be obtained as the asymptotic limit of a screened
electrostatic model. 

The limiting energy arising in Theorem \ref{GC} was mathematically investigated, e.g., in \cite{GSV19} and the corresponding energy-minimizing configurations were also observed experimentally and numerically in \cite{LavrentovichKumariLavrentovich2025,PhysRevResearch.6.043207}.

The model \eqref{Helmholtz}-\eqref{noGL} also bears a formal resemblance to
micromagnetics, where a constrained vector field is coupled to a field
generated by a Maxwell equation. However, the relevant asymptotic regime
is different. The distinction is not merely that the present model involves
a screened, rather than unscreened, field equation. In ferroelectric
nematics, the electrostatic cost associated with bound polarization charge
is large compared with the orientational elastic energy. Thus, in an
unscreened regime one would expect bounded-energy configurations to be
driven toward the charge-free constraints
\[
\operatorname{div} {\bf{P}} = 0 \quad \text{in } \Omega,
\qquad
{\bf{P}}\cdot {\boldsymbol{\nu}} = 0 \quad \text{on } \partial\Omega.
\]
In the screened regime considered here, this same mechanism produces,
after nondimensionalization, a finite local splay penalty in the limit
$\varepsilon \to 0$, together with the boundary constraint
${\bf{P}}\cdot{\boldsymbol{\nu}}=0$. The next section explains this relationship more explicitly
and places the present model in the context of existing ferroelectric-nematic
and micromagnetic models. The subsequent section is devoted to proving
Theorem~\ref{GC}.

\section{Physical motivation and relationship to micromagnetics.}

\subsection{Ferroelectric nematics and high-splay models}

A ferroelectric nematic liquid crystal \cite{ChenEtAl2020FerroelectricNematic} is a nematic liquid with a macroscopic spontaneous polarization ${\bf{P}}$. This polarization is locally parallel to the nematic director ${\bf n}$, which specifies the preferred local molecular orientation \cite{Virga1994VariationalTheories}. The polarization direction can be aligned by confining the material between two glass plates and/or applying an external electric field. The experiments demonstrate formation of domain walls with a shape that is dictated by the anisotropic surface interactions with the “easy axis” of the substrate and by the orientational elasticity of the ferroelectric nematic \cite{PhysRevResearch.6.043207,LavrentovichKumariLavrentovich2025}. Crucially, in this process the splay of ${\bf{P}}$ is energetically costly because it creates a bound space charge of bulk density  proportional to $\dive {\bf{P}}$ and thus increases the electrostatic energy.

Given that ferroelectric nematics have been synthesized very recently, their modeling is yet to be fully developed. The existing work \cite{PhysRevResearch.6.043207} has focused on adapting the well-known director-based approach to standard nematics, based on the Oseen-Frank energy,
\begin{align*}
\label{eq:main}
 E_{OF}({\bf n}):=\int_{\Omega}\frac{K_1}{2}({\rm div}\,{\bf n})^2+&\frac{K_2}{2}({{\bf n}\cdot\rm curl}\,{\bf n})^2+\frac{K_3}{2}\abs{{\bf n}\times{\rm curl}\,{\bf n}}^2,
\end{align*}
that decomposes the elastic energy density into the cost of splay, bend and twist. Because the polarization vector ${\bf{P}}$ tends to align with the director ${\bf{n}}$, one can replace ${\bf{n}}$ with ${\bf{P}}$ when modeling elastic interactions. To fix ideas, consider the set $\Omega\subset\R^3.$ Then for ${\bf{P}}:\Omega\to\R^3$ the (dimensional) elastic energy takes the form
\begin{align}
 E({\bf{P}}):=\int_{\Omega}\frac{K_1}{2}({\rm div}\,{\bf{P}})^2+&\frac{K_2}{2}({{\bf{P}}\cdot\rm curl}\,{\bf{P}})^2+\frac{K_3}{2}\abs{{\bf{P}}\times{\rm curl}\,{\bf{P}}}^2\\ \nonumber &+\frac{\alpha}{4}(\abs{{\bf{P}}}^2-P_0^2)^2\,dx+
\frac{\gamma}{2}\int_{\partial\Omega}({\bf{P}}\cdot{\boldsymbol{\nu}})^4\,dS. \nonumber 
\label{eq:main}
\end{align}
The electrostatic interactions are ignored but the associated high cost of splay is incorporated into the elastic energy density by setting $K_1\gg K_2,K_3.$ The energy is supplemented with the penultimate potential term fixing a preferred value of $P_0$ for the magnitude of the polarization vector. The final anchoring term appearing in \eqref{eq:main} also originates from electrostatics and it takes into account the propensity of polarization to lie in a tangential direction at an interface between different media.

By setting $K_2=K_3=K_1-M=K,$ one can arrive at the following simplified expression for the elastic energy
\begin{align*}
 E({\bf{P}}):=\int_{\Omega}\frac{M}{2}({\rm div}\,{\bf{P}})^2+\frac{K}{2}\abs{\nabla {\bf{P}}}^2&+\frac{\alpha}{4}(\abs{{\bf{P}}}^2-P_0^2)^2\,dx+
\frac{\gamma}{2}\int_{\partial\Omega}({\bf{P}}\cdot{\boldsymbol{\nu}})^4\,dS,   
\end{align*}
where $\alpha,\gamma\gg1$ and $M\gg K>0.$ This version of the model has been shown to capture the experimentally observed features of  ferroelectric nematics when $\Omega=\Omega_0\times[-h/2,h/2]$ with $\Omega_0\subset\R^2$ and $h\ll1$ so that $\Omega$ is a thin film \cite{PhysRevResearch.6.043207}. Briefly, by choosing the parameters of the problem appropriately and nondimensionalizing \cite{PhysRevResearch.6.043207}, we can write the energy $\mathcal{E}_\e$ through a two-component vector field $\mathcal{{\bf{P}}}=\big(\mathcal{{\bf{P}}}^{(1)},\mathcal{{\bf{P}}}^{(2)}\big)$ via
\begin{equation}
 \mathcal{E}_\xi(\mathcal{{\bf{P}}})\sim
\int_{\Omega_0}
    \frac{1}{2\alpha^2}\left({\rm div}\,\mathcal{{\bf{P}}}  \right)^2+\frac{\xi}{2}\abs{\nabla \mathcal{{\bf{P}}}} ^2
    +
    \frac{1}{4\xi}(\abs{\mathcal{{\bf{P}}}}^2-1)^2\,dx.\label{GSV}  
\end{equation}

The two-dimensional energy \eqref{GSV} penalizing splay over bend has been analyzed in \cite{GSV19}. The most salient observation in \cite{GSV19} is that, as the parameter $\xi\to 0$, the minimization problem \eqref{GSV}  approaches a sum of bulk splay cost and a wall cost given by
\begin{equation}
 \mathcal{E}_0(\mathcal{{\bf{P}}}):= \int_{\Omega_0}
    \frac{1}{2\alpha^2}\left({\rm div}\,\mathcal{{\bf{P}}}  \right)^2\,dx\,dy
    +\frac{1}{3}\int_{J_{\mathcal{{\bf{P}}}}}\abs{\mathcal{{\bf{P}}}_+-\mathcal{{\bf{P}}}_-}^3\,ds,\label{GSVgl}
\end{equation}
where $\mathcal{{\bf{P}}}$ has prescribed values on $\partial\Omega_0$ and satisfies $\abs{\mathcal{{\bf{P}}}}=1$ everywhere in $\Omega_0$. The symbol $J_{\mathcal{{\bf{P}}}}$ represents the elastic wall, that is, a curve across which $\mathcal{{\bf{P}}}$ jumps in order to save on the cost of splay and $\mathcal{{\bf{P}}}_+$ and $\mathcal{{\bf{P}}}_-$ denote the values of the polarity on either side of the wall. The minimizers of $\mathcal{E}_0$ for small $\xi$ were shown to be characterized by presence of point singularities and elastic walls. 

If, in addition, one were to impose the condition that $\alpha\ll1,$ then in the unscreened limit $\mathcal{E}_\xi$ can be conjectured to $\Gamma$-converge to the Aviles-Giga energy
\begin{equation}
    \label{eq:AG}
    E_{AG}=\int_\Omega\xi{\left|\nabla^2u\right|}^2+\frac{1}{\xi}{\left({\left|\nabla u\right|}^2-1\right)}^2
\end{equation}
for ${\bf{P}}=:-\nabla^\perp u$ since $\dive {\bf{P}}\to 0$ as $\alpha\to0$ (\cite{aviles1987mathematical}).

Explicitly constructing the critical points of the energy $\mathcal{E}_0$ and computing them numerically via gradient flows for $\mathcal{E}_\e$ reveals singular structures that can be seen in experiments hinting that the electrostatic effects are correctly captured by the reduced energy $E$. However the solid mathematical argument for rigorous reduction from the full model that correctly accounts for electrostatics has so far been lacking. This is the issue we address in the present paper in Theorem \ref{GC}.

\subsection{Screened electrostatics model}
The model accounting for electrostatics appears, for example, in \cite{LavrentovichKumariLavrentovich2025}, \cite{Selinger} and it consists of the (dimensional) Oseen-Frank energy
\begin{equation}
 E_F({\bf{P}})=\frac{1}{2}\int_{\Omega}K_1(\dive {\bf{P}})^2+K_2(\text{curl}\,{\bf{P}}\cdot {\bf{P}})^2+K_3\abs{\text{curl}\,{\bf{P}}\times  {\bf{P}}}^2\,dx,  \label{genOF} 
\end{equation}
for polarization, augmented by the interaction term \[\int_{\Omega}\frac{1}{2}{\bf{P}}\cdot\nabla u\,dx.\] The potential $u$ is assumed to satisfy the dimensional version of \eqref{Helmholtz}-\eqref{bc} given by
\begin{equation}
\Delta u-\frac{2}{\lambda_D^2}u=\left\{\begin{matrix} \frac{1}{\e_0}\dive {\bf{P}}&\text{in}\;\Omega\\ 0&\text{in}\;\R^3\setminus\overline{\Omega}\end{matrix}\right.\label{Helmholtz-dim}
\end{equation}
subject to the boundary conditions
\begin{equation}
[u]=0\quad\text{and}\quad\left[\frac{\partial u}{\partial \nu}\right]=\,-\frac{1}{\e_0}{\bf{P}}\cdot{\boldsymbol{\nu}}\quad\text{along}\;\partial\Omega.\label{bc-dim}
\end{equation}
Here $\lambda_D$ is the Debye length and $\e_0$ is permittivity of the free space. Upon nondimensionalization, and in the regime of equal elastic constants, this leads to the scaled variational problem for \eqref{noGL} with the constraint \eqref{Helmholtz}-\eqref{bc}. One can consider this energy with a more general Oseen-Frank term \eqref{genOF} highlighting the separate costs of splay, twist and bend, either with or without a Ginzburg-Landau penalty. Because our analysis would go through unchanged, provided the Frank constants are taken to satisfy the usual inequalities $K_1,K_2,K_3>0$ so as to ensure ellipticity of the integrand, we choose to work with the simplest choice of the energy functional given by \eqref{noGL}. 

\subsection{Relationship to micromagnetics}

The structure of the ferroelectric-nematic energy is reminiscent of
micromagnetics that have been widely studied in a variational context (see, e.g.,  \cite{GarciaCervera2004OneDimensionalWalls,10.1098/rspa.1997.0013,DeSimone1993EnergyMinimizers,DeSimoneKohnMullerOtto2002ReducedTheory,KohnSlastikov2005AnotherThinFilmLimit,Ignat2009GammaNeel,MoriniMuratovNovagaSlastikov2023TransverseWalls}). In micromagnetics, one studies a magnetization
${\mathbf m}:\Omega\to S^2$ coupled to its magnetostatic field. A typical energy has
the form
\[
E_{\mathrm{mm}}({\mathbf m})
=
d^2\int_\Omega |\nabla {\mathbf m}|^2\,dx
+
\int_\Omega \varphi({\mathbf m})\,dx
+
\frac{\mu_0}{2}\int_{\mathbb R^3}|{\mathbf H}_d|^2\,dx ,
\]
where the demagnetizing field ${\mathbf H}_d$ is curl-free and satisfies
\[
\operatorname{div}({\mathbf H}_d+M_s {\mathbf m}\chi_\Omega)=0
\qquad\text{in }\mathbb R^3.
\]
Writing ${\mathbf H}_d=-\nabla u$, this becomes
\[
\operatorname{div}(-\nabla u+M_s {\mathbf m}\chi_\Omega)=0
\qquad\text{in }\mathbb R^3.
\]

Thus, in both micromagnetics and the ferroelectric-nematic model, a
constrained vector field is coupled to a scalar potential generated by the
divergence of that vector field. This formal similarity should not obscure an important difference in the relevant asymptotic regimes. In micromagnetics, the unscreened magnetostatic interaction gives rise to a variety of limiting behaviors depending on the scaling and domain geometry: in some thin-film
regimes it reduces at leading order to an energetic penalty or a hard constraint on the normal component of magnetization \cite{10.1098/rspa.1997.0013,KohnSlastikov2005AnotherThinFilmLimit}, while
in other regimes the nonlocal contribution is retained in a reduced limit \cite{DeSimoneKohnMullerOtto2002ReducedTheory}.

For ferroelectric nematics the situation is different. The electrostatic field
is generated by bound polarization charge, whose bulk density is proportional
to $-\operatorname{div}{\bf P}$, together with surface charge proportional to
${\bf P}\cdot\boldsymbol{\nu}$. The electrostatic cost associated with these charges is large
relative to the orientational elastic energy. Therefore the distinction from
micromagnetics is not merely that the electrostatic potential in the present
model satisfies a screened Helmholtz equation rather than an unscreened
Poisson equation. In fact, the influence of the electrostatic cost is strongest in the unscreened regime. With $\alpha=0,$ one would expect bounded-energy configurations to be driven toward the charge-free constraints
\[
\operatorname{div}{\bf P}=0
\qquad\text{in }\Omega,
\qquad
{\bf P}\cdot\boldsymbol{\nu}=0
\qquad\text{on }\partial\Omega.
\]
We expect that in this case an analysis analogous to that in the present paper would lead in the $\e\to0$ limit to the Aviles-Giga energy \eqref{eq:AG}. Indeed, the electrostatic interaction in ferroelectric nematics acts more like a strong constraint against splay and normal polarization charge, in comparison to the similar magnetostatic term in the usual micromagnetic scaling.

The screened regime $\alpha>0$ studied in the present paper provides an intermediate
asymptotic mechanism. The Debye-type screening localizes the electrostatic
interaction at length scale $\varepsilon$. As $\varepsilon\to0$, the
bulk part of the screened interaction collapses to a local quadratic penalty
on $\operatorname{div}{\bf P}$, while the surface self-interaction still diverges
unless the normal component of ${\bf P}$ vanishes on the boundary. Thus the
limiting energy contains the local splay contribution
\[
\frac{1}{2\alpha^2}\int_\Omega(\operatorname{div}{\bf P})^2\,dx
\]
and imposes the hard boundary constraint
\[
{\bf P}\cdot\boldsymbol{\nu}=0\qquad\text{on }\partial\Omega.
\]

Consequently, although the model has the formal structure of a micromagnetic
energy, the limiting theory is different. The asymptotic regime relevant to
ferroelectric nematics leads not to a nonlocal magnetostatic energy, but to a
local high-splay Oseen--Frank-type energy together with a tangential boundary
condition. This is the sense in which the present model connects the
electrostatic description of ferroelectric nematics to the high-splay
phenomenology used in reduced Oseen--Frank-type models.

\section{Proof of Theorem \ref{GC}}

We first establish a solution formula for $u_\e$ in terms of integration against the Helmholtz Green's function. While this is presumably quite standard, we present it here since we were unable to find a suitable reference for this precise problem.

Necessarily a solution to \eqref{Helmholtz}-\eqref{bc}, must be considered in the weak sense since jumps in the derivatives are built into the formulation. Thus, we begin with the following result, phrased in terms of more general inhomogeneities for the problem:
\begin{equation}
\e^2\Delta u-\alpha^2u=\left\{\begin{matrix} f&\text{in}\;\Omega\\ 0&\text{in}\;\R^3\setminus\overline{\Omega}\end{matrix}\right.\label{genHelmholtz}
\end{equation}
subject to the boundary conditions
\begin{equation}
[u]=0\quad\text{and}\quad\left[\frac{\partial u}{\partial {\boldsymbol\nu}}\right]=\,{-}\frac{1}{\e^2}h\quad\text{along}\;\partial\Omega,\label{genbc}
\end{equation}
where $f\in L^2(\Omega)$ and $h\in L^2(\partial\Omega)$ are arbitrary functions.
\begin{prop}
Given any $f\in L^2(\Omega)$ and $h\in L^2(\partial\Omega)$, the weak formulation of
\eqref{genHelmholtz}-\eqref{genbc} is given by the condition that a function $u\in H^1(\R^3)$ satisfies
\begin{equation}
\int_{\R^3}\left\{\e^2\nabla u\cdot\nabla \varphi+\alpha^2u\varphi+\chi_{\Omega}f\varphi\right\}\,dx-\int_{\partial\Omega}h\,\varphi\,dS=0\label{weak}
\end{equation}
for any compactly supported, smooth function $\varphi:\R^3\to\R$.
That is, if $u$ satisfies \eqref{weak}, and $u$ is continuous and is sufficiently smooth up to $\partial\Omega$ from either side so as to allow for integration by parts, then $u$ satisfies \eqref{genHelmholtz}-\eqref{genbc}.

In particular, for any vector field ${\bf{P}}\in H^1(\Omega;\R^3)$ the weak formulation of \eqref{Helmholtz}-\eqref{bc} is given by the condition that a function $u\in H^1(\R^3)$ satisfies
\begin{equation}
\int_{\R^3}\left\{\e^2\nabla u\cdot\nabla \varphi+\alpha^2u\varphi+\chi_{\Omega}(\dive {\bf{P}})\varphi\right\}\,dx-\int_{\partial\Omega}{\bf{P}}\cdot{\bf{{\boldsymbol\nu}}}\,\varphi\,dS=0\label{weakp}
\end{equation}
for any compactly supported, smooth function $\varphi:\R^3\to\R$.
\end{prop}
\begin{proof}
Taking arbitrary $\varphi$ supported off of $\partial\Omega$ in \eqref{weak}, an integration by parts immediately leads to \eqref{genHelmholtz}. Then taking $\varphi$ to be supported in a ball centered at a boundary point of $\partial\Omega$, one invokes \eqref{genHelmholtz} after an integration by parts to find that
\[
\int_{\partial\Omega}\e^2\left(\frac{\partial u_-}{\partial {\boldsymbol\nu}}-\frac{\partial u_+}{\partial {\boldsymbol{\nu}}}\right)\varphi\,dS-\int_{\partial\Omega}h\,\varphi\,dS=0
\]
for all such $\varphi$, where $\frac{\partial u_-}{\partial {\boldsymbol{\nu}}}$ denotes the normal derivative in the direction of ${\boldsymbol{\nu}}$ from inside $\Omega$ and $\frac{\partial u_+}{\partial {\boldsymbol{\nu}}}$ denotes the normal derivative in the direction of ${\boldsymbol{\nu}}$ from outside $\Omega$. The jump condition \eqref{genbc} on normal derivatives follows.
\end{proof}

We now recall the fundamental solution in three dimensions associated with the Helmholtz operator $\e^2\Delta \cdot -\alpha^2\cdot$. It is given by
\begin{equation}
\mathcal{G}_\e(x,y):=G_\e(\abs{x-y})\quad\text{where}\; G_\e(r):=\frac{1}{4\pi\e^2r}e^{-\frac{\alpha}{\e}r}.\label{Green}
\end{equation}
As is well-known, $\mathcal{G}_\e$ then satisfies the property
\begin{equation}
\e^2\Delta \mathcal{G}_\e(x,y) -\alpha^2\mathcal{G}_\e(x,y)=\,-\delta(x-y),\label{delta}
\end{equation}
cf. e.g. \cite{Zauderer},  pg. 324.
Here we choose to emphasize the dependence on $\e$ but not $\alpha$ in this notation since we will be interested in the $\e\to 0$ limit of our energy.
Then we have:
\begin{prop}\label{solnform}
Let $f\in L^2(\Omega)$. Then the solution to \eqref{genHelmholtz}-\eqref{genbc} is given by
\begin{equation}
u_\e(x)=-\int_{\Omega}\mathcal{G}_\e(x,y)\,f(y)\,dy+\int_{\partial\Omega} \mathcal{G}_\e(x,y)\,h(y)\,dS_y\label{usoln}
\end{equation}
in the sense that $u_\e$ satisfies \eqref{weak} for any compactly supported smooth function $\varphi:\R^3\to\R$.
\end{prop}
\begin{proof}
Given that $\nabla \mathcal{G}_\e$ has an integrable singularity in three dimensions, it follows that 
\begin{equation}
\nabla u_\e(x)=-\int_{\Omega}\nabla_x \mathcal{G}_\e(x,y)\,f(y)\,dy
+\int_{\partial \Omega}\nabla_x \mathcal{G}_\e(x,y))h(y)\,dS_y
\label{gradu}
\end{equation}
for any $x\not\in\partial\Omega$. Therefore, we may use \eqref{gradu} to write the first two terms on the left-hand side of \eqref{weak} as
\begin{eqnarray}
&&\int_{\R^3}\bigg\{\e^2\nabla u_\e(x)\cdot\nabla \varphi(x)+\alpha^2u_\e(x)\varphi(x)\bigg\}\,dx=\nonumber\\&&
-\int_{\Omega}\int_{\Omega}\bigg\{\e^2\nabla_x \mathcal{G}_\e(x,y)\cdot\nabla\varphi(x)+\alpha^2 \mathcal{G}_\e(x,y)\varphi(x)\bigg\}f(y)\,dy\,dx\nonumber
\\&&+\int_{\Omega}\int_{\partial \Omega}
\bigg\{\e^2\nabla_x \mathcal{G}_\e(x,y)\cdot\nabla\varphi(x)
+\alpha^2\mathcal{G}_\e(x,y)\varphi(x)\bigg\}h(y)\,dS_y\,dx\nonumber\\
&&
-\int_{\Omega^c}\int_{\Omega}\bigg\{\e^2\nabla_x \mathcal{G}_\e(x,y)\cdot\nabla\varphi(x)+\alpha^2 \mathcal{G}_\e(x,y)\varphi(x)\bigg\}f(y)\,dy\,dx\nonumber\\
&&+\int_{\Omega^c}\int_{\partial \Omega}
\bigg\{\e^2\nabla_x \mathcal{G}_\e(x,y)\cdot\nabla\varphi(x)
+\alpha^2 \mathcal{G}_\e(x,y)\varphi(x)\bigg\}h(y)\,dS_y\,dx.\nonumber\\
\label{pre}
\end{eqnarray}
Next we apply Fubini's Theorem to switch the order of integration, integrate by parts and apply \eqref{delta} to the second and fourth lines of \eqref{pre} to obtain
\begin{eqnarray*}
&&\int_{\R^3}\bigg\{\e^2\nabla u_\e(x)\cdot\nabla \varphi(x)+\alpha^2u_\e(x)\varphi(x)\bigg\}\,dx=
-\int_{\Omega}\varphi(x)f(x)\,dx\\
&&+\int_{\partial \Omega}\int_{\Omega}
\left\{\e^2\nabla_x \mathcal{G}_\e(x,y)\cdot\nabla\varphi(x)
+\alpha^2\mathcal{G}_\e(x,y)\varphi(x)\right\}h(y)\,dx\,dS_y\\
&&+\int_{\partial \Omega}\int_{\Omega^c}
\left\{\e^2\nabla_x \mathcal{G}_\e(x,y)\cdot\nabla\varphi(x)
+\alpha^2 \mathcal{G}_\e(x,y)\varphi(x)\right\}h(y)\,dx\,dS_y.\\
\end{eqnarray*}
Here we note that the two new boundary integrals arising from the integration by parts cancel each other out.

Then, to handle the last two integrals, we excise a ball of radius  $\delta$ centered on $\partial \Omega$, integrate by parts, and again appeal to \eqref{delta} to find that
\begin{eqnarray}
&&\int_{\R^3}\bigg\{\e^2\nabla u_\e(x)\cdot\nabla \varphi(x)+\alpha^2u_\e(x)\varphi(x)+
\varphi(x)f(x)\chi_{\Omega}(x)\bigg\}\,dx=\nonumber\\
&&\lim_{\delta\to 0}\bigg[\int_{\partial \Omega}\int_{\Omega\setminus B_\delta(y)}
\left\{\e^2\nabla_x \mathcal{G}_\e(x,y)\cdot\nabla\varphi(x)
+\alpha^2 \mathcal{G}_\e(x,y)\varphi(x)\right\}h(y)\,dx\,dS_y\nonumber\\
&&+\int_{\partial \Omega}\int_{\Omega^c\setminus B_\delta(y)}
\left\{\e^2\nabla_x \mathcal{G}_\e(x,y)\cdot\nabla\varphi(x)
+\alpha^2 \mathcal{G}_\e(x,y)\varphi(x)\right\}h(y)\,dx\,dS_y\bigg]\nonumber\\
&&=\,-\e^2\lim_{\delta\to 0}\int_{\partial\Omega}\int_{\partial B_\delta(y)}\varphi(x)\nabla_x \mathcal{G}_\e\cdot{\boldsymbol{\nu}}_x\,h(y)\,dS_x\,dS_y,\label{almost}
\end{eqnarray}
since the boundary integrals along $\partial\Omega\setminus B_\delta(y)$ and $\partial\Omega^c\setminus B_\delta(y)$ cancel. In the last integral, $\nu_x$ denotes the outer unit normal to the ball $B_\delta(y)$.
Finally, in light of \eqref{Green}, we calculate that for each $y\in\partial\Omega$ one has 
\[\nabla_x \mathcal{G}_\e(x,y)\cdot{\boldsymbol{\nu}}_x=G_\e'(\delta)=\,-\bigg[\frac{\alpha}{\e}\frac{1}{4\pi\e^2\delta} +\frac{1}{4\pi\e^2\delta^2}\bigg]e^{-\frac{\alpha}{\e}\delta}
\quad
\text{for}\; x\in\partial B_\delta(y).
\]
Hence,
\[
\e^2\lim_{\delta\to 0}\int_{\partial\Omega}\int_{\partial B_\delta(y)}\varphi(x)\nabla_x \mathcal{G}_\e\cdot{\boldsymbol{\nu}}_x\,h(y)\,dS_x\,dS_y= 
-\int_{\partial \Omega}\varphi(y)\,h(y)\,dS_y,
\]
and so, returning to \eqref{almost}, we see that $u_\e$ does indeed satisfy \eqref{Helmholtz}-\eqref{bc} in the sense of \eqref{weak}.\\
\end{proof}

A simple fact to be used in the proof of Theorem \ref{GC} is contained in the following:
\begin{lemma}\label{poslem}
    For any positive values of $\e$ and $\alpha$ and for any $f\in L^2(\Omega)$ one has
    \begin{equation}
    \int_\Omega\int_\Omega \mathcal{G}_\e(x,y) f(x)f(y)\,dx\,dy\geq 0.\label{ab}
    \end{equation}
    Similarly, for any $h\in L^2(\partial\Omega)$ one has
   \begin{equation}
    \int_{\partial\Omega}\int_{\partial\Omega} \mathcal{G}_\e(x,y)\, h(x)\,h(y)\,dS_x\,dS_y\geq 0.\label{bb}
    \end{equation} 
\end{lemma}
\noindent{\it Proof of Lemma \ref{poslem}.}
Introducing the function
\[
z_\e(x):=\int_\Omega \mathcal{G}_\e(x,y)f(y)\,dy,
\]
it follows as in the proof of Proposition \ref{solnform} that $z_\e$ satisfies the problem
\begin{equation}
\e^2\Delta z_\e-\alpha^2 z_\e=\left\{\begin{matrix}
    -f(x)&\text{for}\;x\in\Omega\\ 0&\text{for}\;x\in\R^3\setminus \Omega,
\end{matrix}\right.\quad \big[\,z_\e\,\big]=0=\big[\,\frac{\partial z_\e}{\partial\nu}\,\big].\label{peter}   
\end{equation}

Then through an integration by parts one finds that
\begin{eqnarray*}
 &&\int_\Omega\int_\Omega \mathcal{G}_\e(x,y) f(x)f(y)\,dx\,dy=  \int_\Omega z_\e(x)f(x)\,dx\\&&=\int_{\R^3}z_\e(x)\left\{ -\e^2\Delta z_\e+\alpha^2 z_\e\right\} =
 \int_\Omega\cdot+\int_{\R^3\setminus \Omega}\cdot\\
 &&=\int_{\R^3} \e^2\abs{\nabla z_\e}^2+\alpha^2 z_\e^2\,dx,
\end{eqnarray*}
which establishes \eqref{ab}.

To establish \eqref{bb}, we introduce
\[
\omega_\e(y):=\int_{\partial\Omega}G_\e(x,y)\,h(x)\,dS_x,
\]
which by Proposition \ref{solnform} solves the problem
\[
\e^2\Delta \omega_\e-\alpha^2\omega_\e=0\;\text{in}\;\R^3\setminus\partial\Omega,\quad \big[\,\omega_\e\,\big]=0,\;\big[\,\frac{\partial \omega_\e}{\partial\nu}\,\big]=-\frac{1}{\e^2}h\;\text{on}\;\partial\Omega.
\]
Then
\begin{align*}
     &\int_{\partial\Omega}\int_{\partial\Omega} \mathcal{G}_\e(x,y)\, h(x)\,h(y)\,dS_x\,dS_y\\
     &=\int_{\partial\Omega}\omega_\e(y)h(y)\,dS_y=
     -\e^2\int_{\partial\Omega}\omega_\e\left(\frac{\partial\omega_\e^+}{\partial\nu}-\frac{\partial\omega_\e^-}{\partial\nu}\right)\,dS_y\\
     &=\e^2\left\{\int_{\Omega}\dive\big(\omega_\e\nabla\omega_\e\big)\,dy+
     \int_{\Omega^c}\dive\big(\omega_\e\nabla\omega_\e\big)\,dy\right\}\\
     &=\e^2\left\{\int_{\Omega}\frac{\alpha^2}{\e^2}\omega_\e^2+\abs{\nabla\omega_\e}^2\,dy+\int_{\Omega^c}\frac{\alpha^2}{\e^2}\omega_\e^2+\abs{\nabla\omega_\e}^2\,dy\right\}\geq 0.
\end{align*}
\qed\vskip.1in
We also require the following simple estimate:
\begin{lemma}\label{rootlem}
Assume $f\in L^2(\Omega)$ and $h\in L^2(\partial\Omega)$. Then
\begin{equation}
\label{eq:lem2}
\left|\int_{\Omega}\int_{\partial\Omega}\mathcal{G}_\e(x,y)f(x)h(y)\,dS_y\,dx\right|\leq \frac{C}{\sqrt{\e}}{\|f\|}_{L^2(\Omega)}{\|h\|}_{L^2(\partial\Omega)}.
\end{equation}
\end{lemma}
\begin{proof}
We again introduce the function $z_\e:\R^3\to\R$ given by
\[
z_\e(x):=\int_{\Omega}\mathcal{G}_\e(x,y)f(y)\,dy,
\]
which in view of Proposition \ref{solnform} is a solution to the problem \eqref{peter}.

Multiplying this PDE by $z_\e$ and integrating by parts, and using the zero jump conditions, we find that
\begin{eqnarray*}
    &&\int_{\R^3}\e^2\abs{\nabla z_\e}^2+\alpha^2z_\e^2\,dx=\int_{\Omega}z_{\e}f\,dx\\
    &&\leq \frac{\alpha^2}{2}\int_{\Omega}z_\e^2\,dx+\frac{1}{2\alpha^2}\int_{\Omega}f^2\,dx.
\end{eqnarray*}
Hence, 
\[
\norm{z_\e}_{L^2(\Omega)}\leq C{\|f\|}_{L^2(\Omega)}\quad\text{and}\quad\norm{z_\e}_{H^1(\Omega)}\leq \frac{C}{\e}{\|f\|}_{L^2(\Omega)}\quad\text{where}\;C=C\left(\alpha\right).
\]
Then, in the light of interpolation inequality
\[
\norm{F}_{L^2(\partial\Omega)}\leq C\norm{F}_{L^2(\Omega)}^{1/2}\,\norm{F}_{H^1(\Omega)}^{1/2},
\]
which holds for any $F\in H^1(\Omega)$, (cf. \cite{EvansGariepy2015MeasureTheory}), we find that $z_\e$ obeys a bound of the form
\[
\norm{z_\e}_{L^2(\partial\Omega)}\leq \frac{C}{\sqrt{\e}}{\|f\|}_{L^2(\Omega)}.
\]
Then \eqref{eq:lem2} follows from an application of H\"older's inequality.
\end{proof}
\vskip.2in
We turn now to a proof of our $\Gamma$-convergence result.
\vskip.1in
\noindent{\it Proof of Theorem \ref{GC}.}\vskip.1in
\noindent{
\bf The liminf inequality.}
We begin with the assertion of lower-semi-continuity. That is,  we assume that
\begin{equation}
 {\bf{P}}_\e\rightharpoonup {\bf{P}}_0\;\text{weakly in}\; H^1(\Omega;\R^3),\label{weakH1}
\end{equation}
and
we wish to establish the claims:
\begin{equation}
\liminf_{\e\to 0}E_\e({\bf{P}}_\e)\geq E_0({\bf{P}}_0)\quad\text{and}\quad
\liminf_{\e\to 0}\tilde{E}_\e({\bf{P}}_\e)\geq \tilde{E}_0({\bf{P}}_0).\label{lsc}
\end{equation}
In light of Sobolev embedding, we may assume, after passing to a subsequence, that 
\begin{equation}
{\bf{P}}_\e\to {\bf{P}}_0\quad\text{strongly in}\; L^q(\Omega)\quad \text{and}\quad
{\bf{P}}_\e\to {\bf{P}}_0\quad\text{strongly in}\; L^q(\partial\Omega)\;\text{for all}\; q<6.\label{twocons}
\end{equation}
In particular, we note that 
\begin{equation}
    \norm{{\bf{P}}_\e}_{H^1(\Omega;\R^3)}<C_0\quad\text{and}\quad \norm{{\bf{P}}_\e}_{H^{1/2}(\partial\Omega)}<C_0\quad\text{for some}\;C_0>0\;\text{independent of}\;\e.\label{H1bdry}
\end{equation}
As it will come up later in the proof, we recall that a useful characterization of the $H^{1/2}(\partial\Omega)$-norm of a function $h$ is  the sum of the $L^2(\partial\Omega)$-norm along with control of the following semi-norm:
    \begin{align}
\int_{\partial\Omega}\int_{\partial\Omega} \frac{|h(x)-h(y)|^2}{|x-y|^{3}}\,dS_x\,dS_y,\label{Half}
\end{align}
which is therefore uniformly bounded with $h={\bf{P}}_\e$ in light of \eqref{H1bdry},
cf. \cite{DNPV12}.

To begin the proof, we first note that through an appeal to the convexity of the term $\int \abs{\nabla {\bf{P}}_\e}^2$ we may immediately assert that
\begin{equation}
\liminf_{\e\to 0} \int_{\Omega} \abs{\nabla {\bf{P}}_\e}^2\,dx\geq  \int_{\Omega} \abs{\nabla {\bf{P}}_0}^2\,dx\quad\text{and that}\quad \abs{{\bf{P}}_0}=1.\label{easypart}
\end{equation}
Furthermore, in the case of $\tilde{E}_\e$ containing the Ginzburg-Landau potential, cf. \eqref{withGL}, we have
\[
\lim_{\e\to 0}\int_{\Omega}\frac{1}{4\eta^2}(\abs{{\bf{P}}_\e}^2-1)^2\,dx=
\int_{\Omega}\frac{1}{4\eta^2}(\abs{{\bf{P}}_0}^2-1)^2\,dx.
\]
To establish the claim \eqref{lsc} in either the case of $E_\e$ or $\tilde{E}_\e$ it then remains to verify that
\begin{equation}
\liminf_{\e\to 0} \int_{\Omega} \,{\bf{P}}_\e(x)\cdot \nabla u_\e(x)\,dx\geq \left\{\begin{matrix} \frac{1}{\alpha^2}\int_{\Omega}(\dive {\bf{P}}_0(x))^2\,dx&\text{if}\;{\bf{P}}_0\cdot{\boldsymbol{\nu}}=0\\
+\infty&\text{otherwise.}\end{matrix}\right.
\label{toshow}
\end{equation}
To this end, we integrate by parts and substitute in the formula \eqref{usoln} for $u_\e$ to find that
\begin{eqnarray}
&&\int_{\Omega} {\bf{P}}_\e(x)\cdot\nabla u_\e(x)\,dx= -\int_{\Omega}u_\e(x)\dive {\bf{P}}_\e(x)\,dx+\int_{\partial\Omega}u_\e(x){\bf{P}}_\e(x)\cdot{\boldsymbol{\nu}}_x\,dS_x\nonumber\\
&&=\int_{\Omega}\int_{\Omega}\mathcal{G}_\e(x,y)\dive {\bf{P}}_\e(y)\dive {\bf{P}}_\e(x)\,dy\,dx\nonumber\\
&&+\int_{\partial\Omega}\int_{\partial\Omega}\mathcal{G}_\e(x,y){\bf{P}}_\e(y)\cdot{\boldsymbol{\nu}}_y\,{\bf{P}}_\e(x)\cdot{\boldsymbol{\nu}}_x\,dS_y\,dS_x\nonumber\\
&&-2\int_{\Omega}\int_{\partial\Omega}\mathcal{G}_\e(x,y)\dive {\bf{P}}_\e(x){\bf{P}}_\e(y)\cdot{\boldsymbol{\nu}}_y\,dS_y\,dx\nonumber\\
&&=I+II+III.\label{first}
\end{eqnarray}

We will first argue that 
\begin{equation}
\liminf_{\e\to 0} I\geq \frac{1}{\alpha^2}\int_{\Omega}(\dive {\bf{P}}_0(x))^2\,dx.\label{claimone}
\end{equation}
To see this, we note that since 
\begin{equation}
   \int_{\R^3}\frac{1}{4\pi\abs{x}}e^{-\alpha\abs{x}}\,dx=\frac{1}{\alpha^2},\quad\text{it follows from \eqref{Green} that}\quad \int_{\R^3}\mathcal{G}_\e(x,y)\,dy=\frac{1}{\alpha^2},\label{approxid} 
\end{equation}
as well, for every $x$ and every $\e>0$. Since addditionally, $G_\e(r)$ concentrates on a smaller and smaller neighborhood of the origin, it serves as an approximation to the identity. Hence, it follows from standard theory, (see e.g. \cite{WZ}, pg. 148) that for any $L^2(\R^3)$ function $f$ one has
\begin{equation}
\int_{\R^3}\mathcal{G}_\e(x,y)f(y)\,dy\to \frac{1}{\alpha^2}\,f\quad\text{in}\;L^2(\R^3)\;\text{as}\;\e\to 0.\label{standard}
\end{equation} 
Then, writing $I$=
\begin{equation*}
\int_{\Omega}\int_{\Omega}\mathcal{G}_\e(x,y)\dive {\bf{P}}_0(y)\dive {\bf{P}}_\e(x)\,dy\,dx
+\int_{\Omega}\int_{\Omega}\mathcal{G}_\e(x,y)\big(\dive {\bf{P}}_\e(y)-\dive {\bf{P}}_0(y)\big)\dive {\bf{P}}_\e(x)\,dy\,dx,
\end{equation*}
we find that \eqref{weakH1} and \eqref{standard} immediately yield a limit for the first term on the right:
\begin{equation}
\lim_{\e\to 0}\int_{\Omega}\int_{\Omega}\mathcal{G}_\e(x,y)\dive {\bf{P}}_0(y)\dive {\bf{P}}_\e(x)\,dy\,dx=\frac{1}{\alpha^2}\int_{\Omega}\big(\dive {\bf{P}}_0(x)\big)^2\,dx.\label{gammafirst}
\end{equation}
For the second term, we invoke \eqref{ab} to assert that
\begin{eqnarray*}
&&\int_{\Omega}\int_{\Omega}\mathcal{G}_\e(x,y)\big(\dive {\bf{P}}_\e(y)-\dive {\bf{P}}_0(y)\big)\dive {\bf{P}}_\e(x)\,dy\,dx\\
&&=\int_{\Omega}\int_{\Omega}\mathcal{G}_\e(x,y)\big(\dive {\bf{P}}_\e(y)-\dive {\bf{P}}_0(y)\big)\big(\dive {\bf{P}}_\e(x)-\dive {\bf{P}}_0(x)\big)\,dy\,dx\\
&&+\int_{\Omega}\int_{\Omega}\mathcal{G}_\e(x,y)\dive {\bf{P}}_0(x)\big(\dive {\bf{P}}_\e(y)-\dive {\bf{P}}_0(y)\big)\,dy\,dx\geq\\
&& \int_{\Omega}\int_{\Omega}\mathcal{G}_\e(x,y)\dive {\bf{P}}_0(x)\big(\dive {\bf{P}}_\e(y)-\dive {\bf{P}}_0(y)
\big)\,dy\,dx.
\end{eqnarray*}
Appealing again to \eqref{weakH1} and \eqref{standard}, we find that this last integral approaches zero as a product of a strongly convergent sequence and a sequence that weakly converges to zero. Hence, we arrive at the conclusion \eqref{claimone}.

From here, we will split the argument into two cases, depending on whether or not ${\bf{P}}_0\cdot{\boldsymbol{\nu}}= 0$ a.e. on $\partial \Omega$.
\vskip.1in\noindent
{\bf Case 1.}
\[\text{Suppose}\quad {\bf{P}}_0(x)\cdot{\boldsymbol{\nu}}_x\not=0\;\text{on a set of positive measure of}\;\partial \Omega .\]
We will argue that
\begin{equation}
\liminf_{\e\to 0}\; II\geq \frac{C}{\e}\quad\text{while}\quad \abs{III}\leq\frac{C}{\sqrt{\e}},\label{toobig}   
\end{equation}
which, when combined with \eqref{claimone}, will imply \eqref{lsc} since then necessarily,
\begin{equation}
\liminf_{\e\to 0}E_\e({\bf{P}}_\e)=\liminf_{\e\to 0}\tilde{E}_\e({\bf{P}}_\e)=\infty.\label{infty}
\end{equation}

To obtain \eqref{toobig}, we observe that 
\begin{equation}
\int_{\R^2}\e G_\e(\abs{y})\,dy=\frac{1}{2\alpha}\label{2dint}
\end{equation}
for every $\e>0$ and every $\alpha>0$, while concentrating at the origin for small $\e$. Therefore, like $\mathcal{G}_\e$ on $\R^3$, $\e\,\mathcal{G}_\e$ serves as an approximation to the identity on $\R^2$. Defining $h_\e:\partial\Omega\to\R$ via
\begin{equation}
h_\e(x):=\int_{\partial\Omega}\e\,\mathcal{G}_\e(x,y){\bf{P}}_0(y)\cdot{\boldsymbol{\nu}}_y\,dS_y,\label{heps}    
\end{equation}
one then has that 
\begin{equation}
h_\e\to \frac{1}{2\alpha}{\bf{P}}_0(x)\cdot \nu_x\quad\text{in}\;L^2(\partial\Omega)\quad\text{as}\;\e\to 0.\label{bdrycon}
\end{equation}

Consequently, arguing along the lines used above in the analysis of term $I$ we have
\begin{eqnarray}
&&\int_{\partial\Omega}\int_{\partial\Omega}\e\mathcal{G}_\e(x,y){\bf{P}}_\e(y)\cdot{\boldsymbol{\nu}}_y\,{\bf{P}}_\e(x)
\cdot{\boldsymbol{\nu}}_x\,dS_y\,dS_x= \nonumber\\&&\int_{\partial\Omega}\int_{\partial\Omega}\e\mathcal{G}_\e(x,y){\bf{P}}_0(y)\cdot{\boldsymbol{\nu}}_y\,{\bf{P}}_\e(x)\cdot{\boldsymbol{\nu}}_x\,dS_y\,dS_x\nonumber\\
&&+\int_{\partial\Omega}\int_{\partial\Omega}\e\mathcal{G}_\e(x,y)\big({\bf{P}}_\e(y)-{\bf{P}}_0(y)\big)\cdot{\boldsymbol{\nu}}_y\,{\bf{P}}_\e(x)\cdot{\boldsymbol{\nu}}_x\,dS_y\,dS_x\nonumber\\
&&=
\int_{\partial\Omega}\int_{\partial\Omega}\e\mathcal{G}_\e(x,y){\bf{P}}_0(y)\cdot{\boldsymbol{\nu}}_y\,{\bf{P}}_\e(x)\cdot{\boldsymbol{\nu}}_x\,dS_y\,dS_x\nonumber\\
&&+\int_{\partial\Omega}\int_{\partial\Omega}\e\mathcal{G}_\e(x,y)\big({\bf{P}}_\e(y)-{\bf{P}}_0(y)\big)\cdot{\boldsymbol{\nu}}_y\,{\bf{P}}_0(x)\cdot{\boldsymbol{\nu}}_x\,dS_y\,dS_x\nonumber\\
&&+\int_{\partial\Omega}\int_{\partial\Omega}\e\mathcal{G}_\e(x,y)\big({\bf{P}}_\e(y)-{\bf{P}}_0(y)\big)\cdot{\boldsymbol{\nu}}_y\,\big({\bf{P}}_\e(x)-{\bf{P}}_0(x)\big)\cdot{\boldsymbol{\nu}}_x\,dS_y\,dS_x\nonumber\\
&&\geq
\int_{\partial\Omega}\int_{\partial\Omega}\e\mathcal{G}_\e(x,y){\bf{P}}_0(y)\cdot{\boldsymbol{\nu}}_y\,{\bf{P}}_\e(x)\cdot{\boldsymbol{\nu}}_x\,dS_y\,dS_x\nonumber\\
&&+\int_{\partial\Omega}\int_{\partial\Omega}\e\mathcal{G}_\e(x,y)\big({\bf{P}}_\e(y)-{\bf{P}}_0(y)\big)\cdot{\boldsymbol{\nu}}_y\,{\bf{P}}_0(x)\cdot{\boldsymbol{\nu}}_x\,dS_y\,dS_x,\label{long}
\end{eqnarray}
where in the last inequality we invoke \eqref{bb}. Then appealing to \eqref{twocons} and \eqref{bdrycon}, we conclude that
\begin{equation}
\liminf_{\e\to 0}\int_{\partial\Omega}\int_{\partial\Omega}\e\mathcal{G}_\e(x,y){\bf{P}}_\e(y)\cdot{\boldsymbol{\nu}}_y\,{\bf{P}}_\e(x)
\cdot{\boldsymbol{\nu}}_x\,dS_y\,dS_x\geq \frac{1}{2\alpha}\int_{\partial\Omega}({\bf{P}}_0(x)\cdot{\boldsymbol{\nu}}_x)^2\,dS_x,
\label{bigeps}
\end{equation}
from which it follows that $II\geq C/\e$ if 
${\bf{P}}_0(x)\cdot{\boldsymbol{\nu}}_x\not\equiv 0.$

Turning to the estimate of the term $III$ from \eqref{first}, we can apply Lemma \ref{rootlem} with $f=\dive {\bf{P}}_\e$ and $h={\bf{P}}_\e\cdot{\boldsymbol{\nu}}$ to find that
\[
\abs{III}\leq \frac{C}{\sqrt{\e}}\norm{\dive {\bf{P}}_\e}_{L^2(\Omega)}\,\norm{{\bf{P}}_\e\cdot \nu}_{L^2(\partial\Omega)}\leq
\frac{C}{\sqrt{\e}}\,C_0^2,
\]
in light of \eqref{H1bdry}. This finishes the verification of \eqref{toobig} and Case 1 is complete.
\vskip.1in
{\bf Case 2.}
\[
\text{Suppose}\quad {\bf{P}}_0(x)\cdot{\boldsymbol{\nu}}_x=0\;\text{a.e. on}\;\partial \Omega .
\]
If ${\bf{P}}_0\cdot{\boldsymbol{\nu}}=0$, then by \eqref{twocons}, we have that ${\bf{P}}_{\e}\cdot{\boldsymbol{\nu}}\to 0$ in $L^2(\partial\Omega)$. Passing to a further subsequence, if necessary, we can assume that either $\frac{\norm{{\bf{P}}_\e\cdot{\boldsymbol{\nu}}}_{L^2(\partial\Omega)}}{\sqrt{\e}}$ has a finite limit as $\e\to0$ or else it approaches infinity.
\vskip.1in
\noindent{\bf Subcase 2.1}
\begin{equation}
    \text{Suppose that }
\lim_{\e\to 0}\frac{\norm{{\bf{P}}_\e\cdot{\boldsymbol{\nu}}}_{L^2(\partial\Omega)}}{\sqrt{\e}}=0.\label{subtwoone}
\end{equation}
We proceed by returning to the expression \eqref{first} for $\int_{\Omega} {\bf{P}}_\e(x)\cdot\nabla u_\e(x)\,dx$ and replacing each appearance of ${\bf{P}}_\e$ by ${\bf{P}}_0+({\bf{P}}_\e-{\bf{P}}_0)$, yielding:
\begin{eqnarray}
&&\int_{\Omega} {\bf{P}}_\e(x)\cdot\nabla u_\e(x)\,dx= \int_{\Omega}\int_{\Omega}\mathcal{G}_\e(x,y)\dive {\bf{P}}_0(y)\dive {\bf{P}}_0(x)\,dy\,dx\nonumber\\
&&+\int_{\Omega}\int_{\Omega}\mathcal{G}_\e(x,y)\left(\dive {\bf{P}}_\e(y)-\dive {\bf{P}}_0(y)\right)\left(\dive {\bf{P}}_\e(x)-\dive {\bf{P}}_0(x)\right)\,dy\,dx\nonumber\\
&&+2\int_{\Omega}\int_{\Omega}\mathcal{G}_\e(x,y)\dive {\bf{P}}_0(x)\left(\dive {\bf{P}}_\e(y)-\dive {\bf{P}}_0(y)\right)\,dx\,dy\nonumber\\
&&+\int_{\partial\Omega}\int_{\partial\Omega}\mathcal{G}_\e(x,y)\left({\bf{P}}_\e(y)-{\bf{P}}_0(y)\right)\cdot{\boldsymbol{\nu}}_y\,\left({\bf{P}}_\e(x)-{\bf{P}}_0(x)\right)\cdot{\boldsymbol{\nu}}_x\,dS_y\,dS_x\nonumber\\
&&-2\int_{\Omega}\int_{\partial\Omega}\mathcal{G}_\e(x,y)\left(\dive {\bf{P}}_\e(x)-\dive {\bf{P}}_0(x)\right)\left({\bf{P}}_\e(y)-{\bf{P}}_0(y)\right)\cdot{\boldsymbol{\nu}}_y\,dS_y\,dx\nonumber\\
&&-2\int_{\Omega}\int_{\partial\Omega}\mathcal{G}_\e(x,y)\dive {\bf{P}}_0(x){\bf{P}}_\e(y)\cdot{\boldsymbol{\nu}}_y\,dS_y\,dx.
\label{subcaseone}
\end{eqnarray}

Regarding the integral appearing in the third line of \eqref{subcaseone}, we note that
\[
\int_{\Omega}\mathcal{G}_\e(x,y)\dive {\bf{P}}_0(x)\,dx\to \dive {\bf{P}}_0\;\text{in}\;L^2(\Omega),\;\text{and}\; \dive {\bf{P}}_\e\rightharpoonup \dive {\bf{P}}_0\;\text{weakly in}\; L^2(\Omega).\]
Hence, this integral vanishes in the limit $\e\to 0.$
On the other hand, we claim that the combination of the three integrals appearing in the second, fourth and fifth lines of \eqref{subcaseone} is non-negative. Indeed, their sum is simply given by
\[
\int_{\R^3}\e^2\abs{\nabla \tilde{u}_\e}^2+\alpha^2 \tilde{u}_\e^2\,dx,
\]
where $\tilde{u}_\e$ is the solution to \eqref{Helmholtz}--\eqref{bc} for ${\bf{P}}={\bf{P}}_\e-{\bf{P}}_0$.
Hence, after applying \eqref{gammafirst} to the integral on the right-hand side of the first line of \eqref{subcaseone}, we will have established \eqref{toshow} if the integral in the last line of \eqref{subcaseone} vanishes as $\e\to 0$. Applying Lemma \ref{rootlem} with $f=\dive {\bf{P}}_0$ and $h={\bf{P}}_\e\cdot{\boldsymbol{\nu}}$, and appealing to assumption \eqref{subtwoone}, we see that this is indeed the case and Subcase 2.1 is complete.
\vskip.1in
\noindent{\bf Subcase 2.2}
\begin{equation}
    \text{Suppose }\;
\lim_{\e\to 0}\frac{\norm{{\bf{P}}_\e\cdot{\boldsymbol{\nu}}}_{L^2(\partial\Omega)}}{\sqrt{\e}}=\infty.
\label{bignorm}
\end{equation}
Referring back to \eqref{first}, we will argue that 
 $II\gg \abs{III}$  and that $II \to \infty$, again implying \eqref{infty}. We observe that $II$ can be written as 
\begin{multline}
    \label{eq:II}
\int_{\partial\Omega}\int_{\partial\Omega}\mathcal{G}_\e(x,y)\left({\bf{P}}_\e(y)\cdot{\boldsymbol{\nu}}_y\right)\,\left({\bf{P}}_\e(x)\cdot{\boldsymbol{\nu}}_x\right)\,dS_y\,dS_x=\int_{\partial\Omega}\int_{\partial\Omega}\mathcal{G}_\e(x,y){\left({\bf{P}}_\e(x)\cdot{\boldsymbol{\nu}}_x\right)}^2\,dS_y\,dS_x
\\+\int_{\partial\Omega}\int_{\partial\Omega}\mathcal{G}_\e(x,y)\left({\bf{P}}_\e(y)\cdot{\boldsymbol{\nu}}_y-{\bf{P}}_\e(x)\cdot{\boldsymbol{\nu}}_x\right)\,\left({\bf{P}}_\e(x)\cdot{\boldsymbol{\nu}}_x\right)\,dS_y\,dS_x=:II_1+{II}_2.
\end{multline}

To estimate ${II}_1,$ we invoke \eqref{2dint}, which implies that 
\[\int_{\R^2}G_\e(|x|)\,dx=\frac{1}{2\alpha\e},\]
so that
\[\label{eq:II_1est}   \int_{\partial\Omega}\mathcal{G}_\e(x,y)\,dS_y\sim\frac{1}{2\alpha\e},
\]
for a smooth and bounded $\partial\Omega$. It follows that
\begin{equation}
    \label{dmitry}   {II}_1\sim\frac{1}{2\alpha\e}{\|{\bf{P}}_\e\cdot{\boldsymbol{\nu}}\|}^2_{L^2(\partial\Omega)}.
\end{equation}
To estimate ${II}_2,$ we multiply and divide the integrand by ${|x-y|}^{3/2}$ and apply H\"older's inequality to the integral in $dS_y$ for all $x\in\partial\Omega$ to obtain
\begin{multline}\label{eq:IIcomp}{|II}_2|\leq\int_{\partial\Omega}\left|{\bf{P}}_\e(x)\cdot{\boldsymbol{\nu}}_x\right|{\left(\int_{\partial\Omega}\mathcal{G}_\e(x,y)^2{\left|x-y\right|}^3\,dS_y\right)}^{\frac12}\\\times{\left(\int_{\partial\Omega}\frac{{\left({\bf{P}}_\e(y)\cdot{\boldsymbol{\nu}}_y-{\bf{P}}_\e(x)\cdot{\boldsymbol{\nu}}_x\right)}^2}{{\left|x-y\right|}^3}\,dS_y\right)}^{\frac12}\,dS_x.\end{multline}
Because $\partial\Omega$ is smooth and bounded, there exists a constant $C$ which depends only on $\partial\Omega$ when intersecting with any sphere one has the arclength bound
\begin{equation}
    \label{eq:Cs}
    \mathcal{H}^1\left(\partial\Omega\cap B_r(x)\right)\leq Cr
\end{equation}
for all $r>0,$ uniformly in $x\in \partial\Omega$. From \eqref{Green}, \eqref{eq:Cs} and the co-area formula, we obtain
\begin{multline}
    \label{eq:Gest}
    \int_{\partial\Omega}\mathcal{G}_\e(x,y)^2{\left|x-y\right|}^3\,dS_y=\frac{1}{16\pi^2\e^4}\int_{\partial\Omega}{\left|x-y\right|}e^{-\frac{2\alpha}{\e}\left|x-y\right|}\,dS_y\\
    =\frac{C}{\e^4}\int_0^\infty\int_{\partial\Omega\cap\{y:\,\abs{y-x}=r\}}re^{-\frac{2\alpha}{\e}r}\,d\mathcal{H}^1\,dr\\
    \leq \frac{C}{\e^4}\int_0^\infty r^2e^{-\frac{2\alpha}{\e}r}\,dr=
     C\e^{-1},
\end{multline}
for each $x\in\partial\Omega$ and some $C>0$ independent of $\e$. Combining \eqref{eq:IIcomp} and \eqref{eq:Gest} and now using  H\"older's inequality for the integral in $dS_x$ we find that
\begin{multline}
    \label{eq:23}
    |{II}_2|\leq C\e^{-\frac{1}{2}}\int_{\partial\Omega}\left|{\bf{P}}_\e(x)\cdot{\boldsymbol{\nu}}_x\right|{\left(\int_{\partial\Omega}\frac{{\left({\bf{P}}_\e(y)\cdot{\boldsymbol{\nu}}_y-{\bf{P}}_\e(x)\cdot{\boldsymbol{\nu}}_x\right)}^2}{{\left|x-y\right|}^3}\,dS_y\right)}^{\frac12}\,dS_x\\\leq C\e^{-\frac{1}{2}}{\|{\bf{P}}_\e\cdot{\boldsymbol{\nu}}\|}_{L^2(\partial\Omega)}{\left(\int_{\partial\Omega}\int_{\partial\Omega}\frac{{\left({\bf{P}}_\e(y)\cdot{\boldsymbol{\nu}}_y-{\bf{P}}_\e(x)\cdot{\boldsymbol{\nu}}_x\right)}^2}{{\left|x-y\right|}^3}\,dS_y\right)}^{\frac12}\\\leq C\e^{-\frac{1}{2}}{\|{\bf{P}}_\e\cdot{\boldsymbol{\nu}}\|}_{L^2(\partial\Omega)}{\|{\bf{P}}_\e\cdot{\boldsymbol{\nu}}\|}_{H^{1/2}(\partial\Omega)}\leq C\e^{-\frac{1}{2}}{\|{\bf{P}}_\e\cdot{\boldsymbol{\nu}}\|}_{L^2(\partial\Omega)},
\end{multline}
in light of the uniform $H^{1/2}(\partial\Omega)$ bound on ${\bf{P}}_\e\cdot{\boldsymbol{\nu}}$ supplied by \eqref{H1bdry}.

Comparing the estimates \eqref{dmitry} and \eqref{eq:23}, we conclude that ${II}_1\gg\left|{II}_2\right|$ as long as 
\[\frac{1}{\e}{\|{\bf{P}}_\e\cdot{\boldsymbol{\nu}}\|}^2_{L^2(\partial\Omega)}\gg\e^{-\frac{1}{2}}{\|{\bf{P}}_\e\cdot{\boldsymbol{\nu}}\|}_{L^2(\partial\Omega)},\]
i.e., provided that
\[{\|{\bf{P}}_\e\cdot{\boldsymbol{\nu}}\|}_{L^2(\partial\Omega)}\gg\e^{\frac{1}{2}}.\]
Combining this inequality with the condition \eqref{bignorm} guarantees that ${II}_1\gg\left|{II}_2\right|$ when $\e$ is sufficiently small. Hence,
\[
II\sim II_1\sim \frac{1}{2\alpha\e}{\|{\bf{P}}_\e\cdot{\boldsymbol{\nu}}\|}^2_{L^2(\partial\Omega)}.
\]
Using the estimate \eqref{eq:lem2} from Lemma \ref{rootlem} once again, however, shows that $\abs{III}\leq 
\frac{C}{\sqrt{\e}}\|{\bf{P}}_\e\cdot{\boldsymbol{\nu}}\|_{L^2(\partial\Omega)}$, and so $II\gg\left|III\right|$ for a small enough $\e,$ as long as \eqref{bignorm} is satisfied. Since $II\to\infty$ as $\e\to 0$, we obtain \eqref{infty}, completing Subcase 2.2.
\vskip.1in
\noindent{\bf Subcase 2.3}
\begin{equation}
    \text{Suppose that}
\lim_{\e\to 0}\frac{\norm{{\bf{P}}_\e\cdot{\boldsymbol{\nu}}}_{L^2(\partial\Omega)}}{\sqrt{\e}}=C>0.
\label{thresh}
\end{equation}
Here we will proceed in the same manner as in Subcase 2.1, except for the convergence estimate on 
\[\int_{\Omega}\int_{\partial\Omega}\mathcal{G}_\e(x,y)\dive {\bf{P}}_0(x){\bf{P}}_\e(y)\cdot{\boldsymbol{\nu}}_y\,dS_y\,dx.\]
Here we will select an arbitrary $\delta>0$ and a sequence $\{g_n\}\subset C^\infty(\bar\Omega)$ that converges to $\dive {\bf{P}}_0$ in $L^2(\Omega),$ then choose a $k\in\N$ such that ${\|\dive {\bf{P}}_0-g_k\|}_{L^2(\Omega)}<\delta.$ We can write
\begin{multline}
    \label{eq:zub}
    \int_{\Omega}\int_{\partial\Omega}\mathcal{G}_\e(x,y)\dive {\bf{P}}_0(x){\bf{P}}_\e(y)\cdot{\boldsymbol{\nu}}_y\,dS_y\,dx\\=\int_{\Omega}\int_{\partial\Omega}\mathcal{G}_\e(x,y)g_k(x){\bf{P}}_\e(y)\cdot{\boldsymbol{\nu}}_y\,dS_y\,dx\\+\int_{\Omega}\int_{\partial\Omega}\mathcal{G}_\e(x,y)\left(\dive {\bf{P}}_0(x)-g_k(x)\right){\bf{P}}_\e(y)\cdot{\boldsymbol{\nu}}_y\,dS_y\,dx.
\end{multline}
Using \eqref{approxid}, the first integral in \eqref{eq:zub} can be estimated as follows:
\begin{equation}
    \label{eq:zub1} \left|\int_{\Omega}\int_{\partial\Omega}\mathcal{G}_\e(x,y)g_k(x){\bf{P}}_\e(y)\cdot{\boldsymbol{\nu}}_y\,dS_y\,dx\right|\leq\frac{1}{\alpha^2}{\|g_k\|}_{L^\infty(\Omega)}{\|{\bf{P}}_\e\cdot{\boldsymbol{\nu}}\|}_{L^1(\partial\Omega)}\leq C_k{\|{\bf{P}}_\e\cdot{\boldsymbol{\nu}}\|}_{L^2(\partial\Omega)},
\end{equation}
and thus it converges to $0$ as $\e\to0.$ The second integral can be evaluated using Lemma \ref{rootlem}; that is,
\begin{multline}
    \label{eq:zub2}
    \left|\int_{\Omega}\int_{\partial\Omega}\mathcal{G}_\e(x,y)\left(\dive {\bf{P}}_0(x)-g_k(x)\right){\bf{P}}_\e(y)\cdot{\boldsymbol{\nu}}_y\,dS_y\,dx\right|\\\leq \frac{C}{\sqrt{\e}}{\|\dive {\bf{P}}_0-g_k\|}_{L^2(\Omega)}{\|{\bf{P}}_\e\cdot{\boldsymbol{\nu}}\|}_{L^2(\partial\Omega)}\leq C\delta\frac{{\|{\bf{P}}_\e\cdot{\boldsymbol{\nu}}\|}_{L^2(\partial\Omega)}}{\sqrt{\e}}.
\end{multline}
Combining this with \eqref{thresh}, we obtain 
\[\limsup_{\e\to 0}\left|\int_{\Omega}\int_{\partial\Omega}\mathcal{G}_\e(x,y)\left(\dive {\bf{P}}_0(x)-g_k(x)\right){\bf{P}}_\e(y)\cdot{\boldsymbol{\nu}}_y\,dS_y\,dx\right|\leq C\delta,\]
for some $C$ independent of $\delta$. Since $\delta$ is arbitrary, we conclude that 
\[\lim_{\e\to 0}\int_{\Omega}\int_{\partial\Omega}\mathcal{G}_\e(x,y)\left(\dive {\bf{P}}_0(x)-g_k(x)\right){\bf{P}}_\e(y)\cdot{\boldsymbol{\nu}}_y\,dS_y\,dx=0.\]

This concludes the proof of the liminf inequality for $\Gamma$-convergence.
\vskip.1in
\noindent{\bf Recovery sequence}.
To finish the proof of Theorem \ref{GC}, we turn to the construction of a recovery sequence,  say $\{\mathcal{{\bf{P}}}_\e\}$, associated with a given vector field ${\bf{P}}_0\in H^1(\Omega;\R^3)$. This sequence must satisfy the conditions 
\begin{equation}
\mathcal{{\bf{P}}}_\e\rightharpoonup {\bf{P}}_0\;\text{weakly in}\; H^1(\Omega;\R^3)\;\text{as}\;\e\to 0,\label{conv}
\end{equation}
and 
\begin{equation}
    \lim_{\e\to 0}E_\e(\mathcal{{\bf{P}}_\e})=E_0({\bf{P}}_0)\quad\text{or}\quad \lim_{\e\to 0}\tilde{E}_\e(\mathcal{{\bf{P}}_\e})=\tilde{E}_0({\bf{P}}_0),\label{limsup}
\end{equation}
depending on whether we are working with $E_\e$ from \eqref{noGL}, in which case we additionally require $\mathcal{{\bf{P}}}_\e$ to be $\mathbb{S}^2$-valued, or we are working with $\tilde{E}_\e$ from \eqref{withGL}. 

In light of the definitions of $E_0$ or $\tilde{E}_0$, there are two cases to analyze, namely when ${\bf{P}}_0\cdot{\boldsymbol{\nu}}=0$ a.e. on $\partial\Omega$ and when ${\bf{P}}_0\cdot{\boldsymbol{\nu}}\not=0$ on a set of positive measure of $\partial\Omega$. We claim that in either case, the `trivial' recovery sequence will suffice.

Consider first the case where  ${\bf{P}}_0\cdot{\boldsymbol{\nu}}=0$ a.e. on $\partial\Omega$. Choosing $\mathcal{{\bf{P}}}_\e\equiv {\bf{P}}_0$ for each $\e$, conditions \eqref{limsup} both reduce to demonstrating that
\begin{equation}
    \lim_{\e\to 0}\int_{\Omega}{\bf{P}}_0(x)\cdot\nabla u_\e(x)\,dx=\frac{1}{\alpha^2}\int_{\Omega}(\dive {\bf{P}}_0(x))^2\,dx,
\end{equation}
where $u_\e$ is the solution to problem \eqref{Helmholtz}-\eqref{bc} with ${\bf{P}}$ given by ${\bf{P}}_0$.
However, in light of \eqref{first}, this reduces to the limit 
\eqref{gammafirst}, already established.

Finally, in case ${\bf{P}}_0\cdot{\boldsymbol{\nu}}\not=0$ on a set of positive measure of $\partial\Omega$, we claim the trivial sequence $\{\mathcal{{\bf{P}}}_\e\}$ again suffices. Here we must show that
\begin{equation}
 \lim_{\e\to 0}\int_{\Omega}{\bf{P}}_0(x)\cdot\nabla u_\e(x)\,dx=\infty.\label{finally}
\end{equation}
Recalling the decomposition given by \eqref{first} as well as \eqref{gammafirst} with ${\bf{P}}_\e$ replaced by ${\bf{P}}_0$, we observe that
\begin{eqnarray*}
  &&  \int_{\Omega}{\bf{P}}_0(x)\cdot\nabla u_\e(x)\,dx\geq \\ &&
\int_{\partial\Omega}\int_{\partial\Omega}\mathcal{G}_\e(x,y){\bf{P}}_0(y)\cdot{\boldsymbol{\nu}}_y\,{\bf{P}}_0(x)
\cdot{\boldsymbol{\nu}}_x\,dS_y\,dS_x
-2\int_{\Omega}\int_{\partial\Omega}\mathcal{G}_\e(x,y)\dive {\bf{P}}_0(x){\bf{P}}_0(y)\cdot{\boldsymbol{\nu}}_y\,dS_y\,dx.
\end{eqnarray*}
Then claim \eqref{finally} will follow since one can argue that
\[
\int_{\partial\Omega}\int_{\partial\Omega}\mathcal{G}_\e(x,y){\bf{P}}_0(y)\cdot{\boldsymbol{\nu}}_y\,{\bf{P}}_0(x)
\cdot{\boldsymbol{\nu}}_x\,dS_y\,dS_x\geq \frac{C}{\e}
\]
using the same reasoning that led to \eqref{long} and \eqref{bigeps}, while 
\[
\int_{\Omega}\int_{\partial\Omega}\mathcal{G}_\e(x,y)\dive {\bf{P}}_0(x){\bf{P}}_0(y)\cdot{\boldsymbol{\nu}}_y\,dS_y\,dx\leq \frac{C}{\sqrt{\e}},
\]
by Lemma \ref{rootlem}.
\qed

\section{Acknowledgments}
The authors would like to thank Oleg Lavrentovich, Dan Spirn and Dean Louizos for helpful discussions. D.G. acknowledges  support by an NSF grant DMS 2106551. The research of P.S. was supported by a Simons Collaboration grant 585520 and an NSF grant DMS 2106516.

\bibliographystyle{ieeetr}
\bibliography{Newbib}
\end{document}